\documentclass[12pt, reqno]{amsart}
\usepackage{mathrsfs}
\usepackage{amsfonts}
\usepackage{xcolor}
\usepackage{amsmath}
\usepackage{amsthm}
\usepackage{amssymb}
\usepackage{cases}
\usepackage{bm}
\usepackage{cite}
\usepackage[pdfstartview=FitH, colorlinks=true,  linkcolor=red, citecolor=blue]{hyperref}
\numberwithin{equation}{section}
\newtheorem{thm}{Theorem}[section]

\newtheorem{lem}{Lemma}[section]

\newtheorem{remark}{Remark}[section]
\newtheorem{pro}{Proposition}[section]

\newtheorem{theoremalph}{Problem}

\begin{document}
\title[continuity of the $L_p$ Guassian Minkowski problem] {Continuity of the solution to the $L_p$ Minkowski problem in Gaussian probability space* }

\author{Hejun Wang**}

\address{\parbox[l]{1\textwidth}{School of Mathematics and
Statistics, Shandong Normal University, Ji'nan, Shandong 250014, China}}
\email{wanghjmath@sdnu.edu.cn}


\thanks{*Supported in part by China Postdoctoral Science Foundation (No.2020M682222)
and Natural Science Foundation of Shandong (No.ZR2020QA003, No.ZR2020QA004)}
\thanks{**The corresponding author}

\maketitle

\begin{abstract}
In this paper, it is proved that the weak convergence of the $L_p$ Guassian
surface area measures implies the convergence of the corresponding convex bodies in the
Hausdorff metric for $p\geq 1$. Moreover, this paper obtains the solution to the $L_p$ Guassian Minkowski
problem is continuous with respect to $p$.
\end{abstract}

\vskip 0.5cm
\section{Introduction}
\vskip 0.3cm

A compact convex set with non-empty interior is called a convex body in $n$-dimensional Euclidean space $\mathbb{R}^n$.
Let $\mathcal{K}^n$ denote the set of all convex bodies in $\mathbb{R}^n$, and let $\mathcal{K}_o^n$ denote the set of all convex bodies in $\mathbb{R}^n$ containing the origin in their interiors. It is clear that $\mathcal{K}_o^n$ is subset of $\mathcal{K}^n$.

The Brunn-Minkowski theory is the very core of convex geometric analysis, and Minkowski problem is one of main parts of the Brunn-Minkowski theory. Minkowski problem characterizes a geometric measure generated by convex bodies.
For smooth case, it corresponds to Monge-Amp\`{e}re type equation in partial differential equations.
The study of Minkowski problem promotes greatly
developments of the Brunn-Minkowski theory \cite{Schneider}) and fully non-linear partial differential equations (see \cite{TrudingerWang}).

In 1990s, Lutwak \cite{Lutwak3} introduced the $L_p$ surface area measure $S_{p}(K, \cdot)$ of convex body $K\in \mathcal{K}_o^n$ which is a fundamental concepts in convex geometric analysis defined  by the variational formula of the $n$-dimensional volume (Lebesgue measure) $V_n$ as follows:

For $p\in \mathbb{R}\setminus\{0\}$,
\begin{align}\label{Lp-measure}
\lim_{t\rightarrow0^+}\frac{V_n(K+_pt\cdot L)-V_n(K)}{t}
=\frac{1}{p}\int_{S^{n-1}}h_L^p(u)dS_{p}(K,u),
\end{align}
where $K+_pt\cdot L$ is the $L_p$ Minkowski combination of $K,L\in \mathcal{K}_o^n$ (see \eqref{Lp-Min-comb}), and $h_L$ is the support function of $L$ on the unit sphere $S^{n-1}$ (see \eqref{support-function}). Note that the case for $p=0$ can be defined by similar way.
When $p=1$, the $L_1$ surface area measure $S_1(K,\cdot)$ is the well-known classical surface area measure $S_K$, that is, $S_1(K,\cdot)=S_K$.

The Minkowski problem for the $L_p$ surface area measure
is called $L_p$ Minkowski problem as follows:
\vskip 0.2cm
\noindent
{\bf $L_p$ Minkowski problem:}
\emph{For a fixed $p$ and a given non-zero finite Borel measure $\mu$ on $S^{n-1}$, what are the necessary and sufficient conditions on $\mu$ such that
there exists a convex body $K$ in $\mathbb{R}^n$ such that its
$L_p$ surface area measure $S_p(K,\cdot)$ is equal to $\mu$, that is,
\begin{align*}
\mu=S_p(K,\cdot)?
\end{align*}}

When $p=1$, the $L_p$ Minkowski problem is the classical Minkowski problem studied by Minkowski \cite{Minkowski1897,Minkowski1903}, Alexandrov \cite{Aleksandrov1938,Aleksandrov1939}, Fenchel-Jensen \cite{FenchelJ} and others. Besides, the centro-affine Minkowski problem ($p=-n$) and the logarithmic
Minkowski problem ($p=0$) are other two special cases of the $L_p$ Minkowski problem, see \cite{ChouWang,BoroczkyLYZ2,BoroczkyHZ,Zhu1,Zhu3,JianLuZhu,Stancu1,Stancu2}.
For the existence, uniqueness and regularity of the (normalized) $L_p$ Minkowski problem, one can see \cite{Lutwak3,LutwakOliker,LYZ3,
HuangLiuXu,LuWang,Zhu2,Zhu5,HugLYZ,Chen}.
As an important application, the solutions to the $L_p$ Minkowski problem play a vital role in discovering some new (sharp) affine $L_p$ Sobolev inequalities,
see \cite{Zhang1,LYZ2,CianchiLYZ,HaberlSchuster1,HaberlSchuster2,HaberlSchusterXiao,Wang1}.
Besides, the $L_p$ Minkowski problem was studied by some curvature flows.

The $L_p$ surface area measure has an important property as follows:
\emph{If the sequence $\{K_i\}\subseteq \mathcal{K}_o^n$ converges  to $K_0\in \mathcal{K}_o^n$ in the Hausdorff metric, then $\{S_{p}(K_i,\cdot)\}$ converges to $S_{p}(K_0,\cdot)$ weakly.}
This is the continuity of the $L_p$ surface area measure with respect to the Hausdorff metric. The reverse question of this result is interesting:
\begin{theoremalph}\label{A}
Does the sequence $\{K_i\}\subseteq \mathcal{K}_o^n$ converge  to $K_0\in \mathcal{K}_o^n$ in the Hausdorff metric as $\{S_{p}(K_i,\cdot)\}$ converges to $S_{p}(K_0,\cdot)$ weakly?
\end{theoremalph}
This problem is related closely to the $L_p$ Minkowski problem and can be restated as
follows:
\begin{theoremalph}\label{B}
Suppose that $p\in \mathbb{R}$, $K_i\in \mathcal{K}_o^n$ is the solution to the $L_p$ Minkowski problem associated with Borel measure $\mu_i$ on $S^{n-1}$ and $K_0\in \mathcal{K}_o^n$ is the solution to the $L_p$ Minkowski problem associated with Borel measure $\mu_0$ on $S^{n-1}$.
Does the sequence $\{K_i\}$ converge  to $K_0$ in the Hausdorff metric as $\{\mu_i\}$ converges to $\mu_0$ weakly?
\end{theoremalph}

In this sense,
Problem \ref{A} (or Problem \ref{B}) is called continuity of the solution to the $L_p$ Minkowski problem.
Since the $L_p$ surface area measure is positively homogeneous of degree $n-p$, then Zhu \cite{Zhu4} showed that this question is not positive for $p=n$ by the following counterexample:
\emph{Let $p=n$, $K_1$ be an origin-symmetric convex body and $K_i=\frac{1}{i}K_1$, then
$S_n(K_i,\cdot)=S_n(K_1,\cdot)$ for
all $i$ but $\{K_i\}$ converges to the origin as $i\rightarrow +\infty$.}
 Moreover, Zhu \cite{Zhu4}  gave an affirmative answer to Problem \ref{A} for $p>1$ with $p\neq n$.
The part results of Problem \ref{A} for $p=0$ and $0<p<1$ were obtained in \cite{WangLv,WFZ2019-2}

In 2016, Huang-Lutwak-Yang-Zhang \cite{HuangLYZ} defined the dual curvature measure by the variational formula of the dual volume (see \cite{Lutwak1}) for $L_1$ Minkowski combination
 and studied the corresponding Minkowski problem called dual Minkowski problem.
The existence, uniqueness and regularity of this problem and its generalization were studied in
\cite{BoroczkyF,BoroczkyHP,BoroczkyLYZZ,LYZ4,WJ2020,Zhao1,Zhao2,ZhuXY2018}.
The continuity of the solution to this problem for $q<0$ was studied in \cite{WFZ2019-1}.

Recently, the Brunn-Minkowski theory for the Gaussian probability measure $\gamma_n$ has hot attention defined by
\begin{align*}
\gamma_n(E)=\frac{1}{(\sqrt{2\pi})^n}\int_{E}e^{-\frac{|x|^2}{2}}dx,
\end{align*}
where $E$ is a subset of $\mathbb{R}^n$ and $|x|$ is the absolute value of $x\in E$. $\gamma_n(E)$ is called the Gaussian volume of $E$.
Since the Gaussian volume $\gamma_n$ does not have translation invariance and homogeneity, then there are more difficulties to study the corresponding Brunn-Minkowski theory. This makes Brunn-Minkowski theory for $\gamma_n$ quite mysterious and further stimulates people's interest.
The Brunn-Minkowski inequality and the Minkowski inequality for the Gaussian volume $\gamma_n$ were studied in \cite{EskenazisG,BoroczkyK,Saroglou,GardnerZ}.

By the variational formula of the Gaussian volume $\gamma_n$  for $L_p$ Minkowski combination,
the $L_p$ Gaussian surface area measure $S_{p,\gamma_n}(K,\cdot)$ of convex body $K\in \mathcal{K}_o^n$
is defined in \cite{HuangXZ,Liu,LvWang} by

For $K,L\in \mathcal{K}^n_o$ and $p\neq 0$,
\begin{align}\label{Lp-Gaussian-Measure}
\lim_{t\rightarrow0}\frac{\gamma_n(K+_pt\cdot L)-\gamma_n(K)}{t}
=\frac{1}{p}\int_{S^{n-1}}h_L^p(u)dS_{p,\gamma_n}(K,u).
\end{align}
When $p=1$, it is the Gaussian surface area measure defined in \cite{HuangXZ}, that is,
$S_{1,\gamma_n}(K,\cdot)=S_{\gamma_n,K}$.
Note that the $L_p$ Gaussian surface area measure $S_{p,\gamma_n}(K,\cdot)$ is not positively homogeneous.

The corresponding Minkowski problem in Gaussian probability space is called the $L_p$ Gaussian Minkowski problem (see \cite{HuangXZ,Liu,LvWang}) as follows:
\vskip 0.2cm
\noindent
{\bf $L_p$  Gaussian Minkowski problem:}
\emph{For fixed $p$ and a given non-zero finite Borel measure $\mu$ on $S^{n-1}$, what are the necessary and sufficient conditions on $\mu$  in order that
there exists a convex body $K\in \mathcal{K}^n_o$  such that
\begin{align*}
\mu=S_{p,\gamma_n}(K,\cdot)?
\end{align*}
}
If $f$ is the density of the given measure $\mu$,
then the corresponding Monge-Amp\`{e}re type equation on $S^{n-1}$ is as follows:

For $u\in S^{n-1}$,
\begin{align*}
    \frac{1}{(\sqrt{2\pi})^n}e^{-\frac{|\nabla h(u)|^2+h^2(u)}{2}}
h^{1-p}(u)\text{det}(\nabla^2 h(u)+h(u)I)=f(u),
\end{align*}
where $h:S^{n-1}\rightarrow (0, +\infty)$ is the function to be found, $\nabla h, \nabla^2 h$ are the gradient vector and the Hessian matrix of $h$ with respect to an orthonormal frame on $S^{n-1}$, and $I$ is the identity matrix.

In this paper, we mainly consider the continuity of the solution to the $L_p$ Gaussian Minkowski problem and obtain the following result:
\begin{thm}\label{thm0-1}
Suppose $p\geq1$ and $K_i\in\mathcal{K}^n_o$ with $\gamma_n(K_i)\geq1/2$ for $i=0,1,2,\cdots$.
If the sequence $\{S_{p,\gamma_n}(K_i,\cdot)\}$ converges to $S_{p,\gamma_n}(K_0,\cdot)$ weakly, then the sequence $\{K_i\}$ converges to $K_0$ in the Hausdorff metric.
\end{thm}

Besides,  we obtain that the solution to the $L_p$ Guassian Minkowski
problem is continuous with respect to $p$.

\begin{thm}\label{thm0-2}
Suppose $p_i\geq1$ and $K_i\in\mathcal{K}^n_o$ with $\gamma_n(K_i)\geq1/2$ for $i=0,1,2,\cdots$.
If $S_{p_i,\gamma_n}(K_i,\cdot)=S_{p_0,\gamma_n}(K_0,\cdot)$, then the sequence $\{K_i\}$ converges to $K_0$ in the Hausdorff metric  as $\{p_i\}$ converges to $p_0$.
\end{thm}

\vskip 0.5cm
\section{Preliminaries}\label{Preliminaries}
\vskip 0.3cm

In this section, we list some notations and recall some basic facts about convex bodies.
According to the context of this paper, $|\cdot|$ can denote different meanings: the absolute value and the total mass of a finite measure.
For vectors $x,y\in\mathbb{R}^n$, $x\cdot y$ denotes the standard inner product in $\mathbb{R}^n$. $S^{n-1}$ denotes the boundary of the Euclidean unit
ball $B_n=\{x\in\mathbb{R}^n: \sqrt{x\cdot x}\leq 1\}$ and is called unit sphere. Let $\omega_n$ denote the $n$-dimensional volume (Lebesgue measure) of $B_n$. Let $\partial K$ and $\text{int}~K$ denote
the boundary and the set of all interiors of convex body $K$ in $\mathbb{R}^n$, respectively.
$\partial' K$ is the subset of $\partial K$ with unique outer unit normal.

The
support function $h_K: \mathbb{R}^n\rightarrow \mathbb{R}$ of $K\in \mathcal{K}^n$ is defined  by
\begin{align}\label{support-function}
h_K(x)=\max\{x\cdot y: y\in K\},\quad x\in \mathbb{R}^n.
\end{align}
A convex body is uniquely determined by its support function.
Support functions are positively homogeneous of degree one and subadditive.
For $K\in \mathcal{K}^n_o$, its support function $h_K$ is continuous and strictly positive on the unit sphere $S^{n-1}$.

The radial function $\rho_K: \mathbb{R}^n\setminus\{0\}\rightarrow \mathbb{R}$ of convex body $K\in \mathcal{K}^n_o$ is another important function for $K\in \mathcal{K}^n_o$, and it is given by
\begin{align*}
\rho_K(x)=\max\{\lambda>0: \lambda x\in K\},\quad x\in \mathbb{R}^n\setminus\{0\}.
\end{align*}
Note that the radial function $\rho_K$
of $K\in \mathcal{K}^n_o$ is positively homogeneous of degree $-1$, and it is continuous and strictly positive on the unit sphere $S^{n-1}$.
For each $u\in S^{n-1}$, $\rho_K(u)u\in\partial K$.

The set $\mathcal{K}^n_o$ can be endowed with Hausdorff metric and radial metric which mean the distance between two convex bodies.
The Hausdorff metric of $K, L\in \mathcal{K}^n_o$ is defined by
\begin{align*}
|| h_K-h_L ||=\mathop{\max}\limits_{u\in S^{n-1}}|h_K(u)-h_L(u)|.
\end{align*}
The radial metric of $K, L\in \mathcal{K}^n_o$
is defined by
\begin{align*}
|| \rho_K-\rho_L ||=\mathop{\max}\limits_{u\in S^{n-1}}|\rho_K(u)-\rho_L(u)|.
\end{align*}
The two metrics are mutually equivalent, that is, for $K, K_i\in \mathcal{K}^n_o$,
\begin{align*}
h_{K_i}\rightarrow h_K~\text{uniformly}\quad \text{if and only if} \quad  \rho_{K_i}\rightarrow \rho_K~\text{uniformly}.
\end{align*}
If $|| h_{K_i}-h_K ||\rightarrow0$ or $||\rho_{K_i}-\rho_K ||\rightarrow0$ as $i\rightarrow+\infty$, we call the sequence $\{K_{i}\}$ converges to $K$.

The polar body $K^*$ of $K\in \mathcal{K}^n_o$
is given by
\begin{align*}
K^*=\{ x\in\mathbb{R}^n: x\cdot y\leq1~\text{for~all}~y\in K\}.
\end{align*}
It is clear that $K^*\in \mathcal{K}^n_o$ and $K=(K^{*})^*$. There exists  an important fact on $\mathbb{R}^n\setminus\{0\}$ between $K$ and its polar body $K^*$:
\begin{align*}
h_K=1/\rho_{K^*}\quad \text{and}\quad \rho_K=1/h_{K^*}.
\end{align*}
Then, for $K,K_i\in \mathcal{K}^n_o$, we can obtain the following result:
\begin{align*}
K_i\rightarrow K\quad \text{if and only if} \quad  K^*_i\rightarrow K^*.
\end{align*}

For $f\in C^+(S^{n-1})$, the Wullf shape $[f]$ of $f$ is defined by
\begin{align*}
[f]=\{ x\in\mathbb{R}^n: x\cdot u\leq f(u)~\text{for~all}~ u\in S^{n-1}\}.
\end{align*}
It is not hard to see that $[f]$ is a convex body in $\mathbb{R}^n$ and $h_{[f]}\leq f$. In addition, $[h_K]=K$ for all $K\in \mathcal{K}^n_o$.

By the concept of Wullf shape, the $L_p$ Minkowski combination can be defined for all $p\in\mathbb{R}$. When $p\neq 0$, for
$K,L\in \mathcal{K}^n_o$ and $s,t\in\mathbb{R}$ satisfying that
 $sh_K^p+th_L^p$ is strictly positive on $S^{n-1}$, the $L_p$ Minkowski combination $s\cdot K+_pt\cdot L$ is defined by
 \begin{align}\label{Lp-Min-comb}
 s\cdot K+_pt\cdot L=[(sh_K^p+th_L^p)^{1/p}].
 \end{align}
When $p=0$, the $L_p$ Minkowski combination $s\cdot K+_0t\cdot L$ is defined by
 \begin{align*}
 s\cdot K+_0t\cdot L=[h_K^sh_L^t].
 \end{align*}

\vskip 0.5cm
\section{The proof of Theorem \ref{thm0-1}}\label{}
\vskip 0.3cm

In this section, we consider the continuity of the solution to the $L_p$ Gaussian Minkowski problem
and obtain the result for $p\geq1$.
By the variational formula \eqref{Lp-Gaussian-Measure} and Ehrhard inequality, the following Minkowski-type inequality is obtained in \cite{HuangXZ,Liu}.

\begin{lem}[\cite{HuangXZ}]\label{Conti-Min-ineq}
Suppose $K,L\in\mathcal{K}_o^n$, then, for $p\geq 1$,
\begin{align*}
\frac{1}{p}\int_{S^{n-1}}h_L^p(u)-h^p_K(u)dS_{p,\gamma_n}(K,u)\geq\gamma_n(K)\log\frac{\gamma_n(L)}{\gamma_n(K)},
\end{align*}
with equality if and only if $K=L$.
\end{lem}

The following lemmas will be needed.

\begin{lem}\label{Conti-cos-bound}
Suppose $p\geq 1$ and $K_i\in \mathcal{K}^n_o$ for $i=0,1,2,\cdots$.
If the sequence $\{S_{p,\gamma_n}(K_i,\cdot)\}$ converges to $S_{p,\gamma_n}(K_0,\cdot)$ weakly, then there exists a constant $c_1>0$ such that
\begin{align}
\int_{S^{n-1}}(u\cdot v)_+^pdS_{p,\gamma_n}(K_i,v)\geq c_1,
\end{align}
for all $u\in S^{n-1}$ and $i\in\{0,1,2,\cdots\}$, where, $(u\cdot v)_+=\max\{0,u\cdot v\}$.
\end{lem}

\begin{proof}
For $x\in \mathbb{R}^{n}$, let
$$
g_i(x)=\int_{S^{n-1}}(x\cdot v)_+^pdS_{p,\gamma_n}(K_i,v).
$$
It is not hard to see that $g_i$ is a sublinear function on $\mathbb{R}^{n}$.
Then, there exists a convex body in $\mathbb{R}^{n}$ such that its support function
is equal to $g_i$.
By the fact that $\{S_{\gamma_n,K_i}\}$ converges to $S_{\gamma_n,K_0}$ weakly, we have the sequence $\{g_i\}$ converges to $g_0$ pointwise.
Since pointwise and uniform convergence of support functions are equivalent on $S^{n-1}$,
then $\{g_i\}$ converges to $g_0$ on $S^{n-1}$ uniformly.

Since $S_{p,\gamma_n}(K_0,\cdot)$ is not concentrated in any closed hemisphere of $S^{n-1}$, then $g_0>0$ on $S^{n-1}$.
Together with the compactness of $S^{n-1}$ and the continuity of $g_0$ on $S^{n-1}$,
we obtain there exists a constant $c_2>0$ such that
\begin{align*}
g_0(u)\geq c_2,
\end{align*}
for all $u\in S^{n-1}$.
Since $\{g_i\}$ converges to $g_0$ uniformly on $S^{n-1}$ , then there exists a constant $c_1>0$ such that
\begin{align*}
g_i(u)\geq c_1,
\end{align*}
for all $u\in S^{n-1}$ and $i=0,1,2,\cdots$.
\end{proof}

The weak convergence of the $L_p$ Guassian
surface area measures implies the sequence of the corresponding convex bodies
is bounded.

\begin{lem}\label{Conti-bound}
Suppose $p\geq 1$ and $K_i\in \mathcal{K}^n_o$ with  $\gamma_n(K_i)\geq1/2$ for $i=0,1,2,\cdots$.
If the sequence $\{S_{p,\gamma_n}(K_i,\cdot)\}$ converges to $S_{p,\gamma_n}(K_0,\cdot)$ weakly, then the sequence $\{K_i\}$ is bounded.
\end{lem}

\begin{proof}
For $K\in\mathcal{K}^n_o$, the function $\Phi_p$ is defined by
\begin{align}\label{Phi}
\Phi_p(K)=-\frac{1}{p\gamma_n(K)}\int_{S^{n-1}}h^p_K(u)dS_{p,\gamma_n}(K,u)+\log\gamma_n(K).
\end{align}
From Lemma \ref{Conti-Min-ineq} and $\gamma_n(K_i)\geq1/2$, we have
\begin{align*}
\Phi_p(K_i)&=-\frac{1}{p\gamma_n(K_i)}\int_{S^{n-1}}
h^p_{K_i}(u)dS_{p,\gamma_n}(K_i,u)+\log\gamma_n(K_i)\\
&\geq-\frac{1}{p\gamma_n(K_i)}\int_{S^{n-1}}h^p_{B_n}(u)dS_{p,\gamma_n}(K_i,u)+\log\gamma_n(B_n)\\
&=-\frac{|S_{p,\gamma_n}(K_i,\cdot)|}{p\gamma_n(K_i)}+\log\gamma_n(B_n)\\
&\geq-\frac{2}{p}|S_{p,\gamma_n}(K_i,\cdot)|+\log\gamma_n(B_n).
\end{align*}

Since the sequence $\{S_{p,\gamma_n}(K_i,\cdot)\}$ converges to $S_{p,\gamma_n}(K_0,\cdot)$ weakly,
then, $\{|S_{p,\gamma_n}(K_i,\cdot)\}$ converges to $|S_{p,\gamma_n}(K_0,\cdot)$. Together with the fact that $|S_{\gamma_n,K_0}|$ is finite, there exists a big enough constant $M_1>0$ satisfying
\begin{align*}
|S_{p,\gamma_n}(K_i,\cdot)|\leq \frac{p}{2}M_1,
\end{align*}
for all $i$.
Hence, for $i=1,2,\cdots$,
\begin{align}\label{Conti-bound-1}
\Phi_p(K_i)\geq-M_1+\log\gamma_n(B_n).
\end{align}

Since $K_i\in\mathcal{K}^n_o$, then $\rho_{K_i}$ is continuous on $S^{n-1}$. By the fact that
$S^{n-1}$ is compact, we obtain there exists $u_i\in S^{n-1}$ such that
\begin{align*}
\rho_{K_i}(u_i)=\max\{\rho_{K_i}(u): u\in S^{n-1}\}.
\end{align*}
Let $R_i=\rho_{K_i}(u_i)$.
then, $R_iu_i\in K_i$ and $K_i\subseteq R_iB_n$.
By the definition of support function,
\begin{align*}
    h_{K_i}(v)\geq R_i(u_i\cdot v)_+,
\end{align*}
for all $v\in S^{n-1}$.
Combining $p\geq1$, $K_i\subseteq R_iB_n$ with Lemma \ref{Conti-cos-bound}, we have
\begin{align*}
\Phi_p(K_i)&=-\frac{1}{p\gamma_n(K_i)}\int_{S^{n-1}}
h^p_{K_i}(v)dS_{p,\gamma_n}(K_i,v)+\log\gamma_n(K_i)\\
&\leq-\frac{R_i^p}{p\gamma_n(K_i)}\int_{S^{n-1}}(u_i\cdot v)_+^pdS_{p,\gamma_n}(K_i,v)+\log\gamma_n(K_i)\\
&\leq-\frac{R_i^p}{p\gamma_n(R_iB_n)}\int_{S^{n-1}}(u_i\cdot v)_+^pdS_{p,\gamma_n}(K_i,v)+\log\gamma_n(R_iB_n)\\
&\leq-\frac{c_1R_i^p}{p\gamma_n(R_iB_n)}+\log\gamma_n(R_iB_n).
\end{align*}

Suppose that $\{R_i\}$ is not a bounded sequence. Without loss of generality,
we may assume that $\lim_{i\rightarrow+\infty}R_i=+\infty$.
By polar coordinates,
\begin{align*}
\gamma_n(R_iB_n)=\frac{1}{(\sqrt{2\pi})^n}\int_{R_iB_n}e^{-\frac{|x|^2}{2}}dx
=\frac{n\omega_n}{(\sqrt{2\pi})^n}\int_0^{R_i}t^{n-1}e^{-\frac{t^2}{2}}dt.
\end{align*}
By the fact that $n\omega_n=2\pi^\frac{n}{2}\big/\Gamma(\frac{n}{2})$ where $\Gamma(s)=\int_0^{+\infty}t^{s-1}e^{-t}dt$ is the Gamma function for $s>0$,
\begin{align}\label{gamma=1}
\lim_{i\rightarrow+\infty}\gamma_n(R_iB_n)
&=\frac{n\omega_n}{(\sqrt{2\pi})^n}\int_0^{+\infty}t^{n-1}e^{-\frac{t^2}{2}}dt\nonumber\\
&=\frac{n\omega_n 2^{\frac{n}{2}-1}}{(\sqrt{2\pi})^n}
\int_0^{+\infty}\left(\frac{t^2}{2}\right)^{\frac{n}{2}-1}e^{-\frac{t^2}{2}}tdt\nonumber\\
&=\frac{n\omega_n 2^{\frac{n}{2}-1}}{(\sqrt{2\pi})^n}
\int_0^{+\infty}t^{\frac{n}{2}-1}e^{-t}dt\nonumber\\
&=\frac{2\pi^\frac{n}{2} 2^{\frac{n}{2}-1}}{\Gamma(\frac{n}{2})(\sqrt{2\pi})^n}\Gamma(\frac{n}{2})=1.
\end{align}
Together with $\lim_{i\rightarrow+\infty}R_i=+\infty$ and $p\geq 1$,
we have
\begin{align*}
\Phi_p(K_i)\leq-\frac{c_1R_i^p}{\gamma_n(R_iB_n)}+\log\gamma_n(R_iB_n)
\rightarrow -\infty,
\end{align*}
as $i\rightarrow+\infty$.
This is a contradiction to \eqref{Conti-bound-1}.
Therefore, $\{R_i\}$ is bounded, that is,
the sequence $\{K_i\}$ is bounded.
\end{proof}

\begin{lem}\label{interior}
Suppose $K$ is a compact convex set in $\mathbb{R}^n$. If $\gamma_n(K)>0$, then $K$ is a convex body in $\mathbb{R}^n$, that is, $K\in \mathcal{K}^n$.
\end{lem}

\begin{proof}
By the definition of the Gaussian volume $\gamma_n$, we have
\begin{align*}
\gamma_n(K)=\frac{1}{(\sqrt{2\pi})^n}\int_Ke^{-\frac{|x|^2}{2}}dx
\leq\frac{1}{(\sqrt{2\pi})^n}\int_Kdx=\frac{1}{(\sqrt{2\pi})^n}V_n(K).
\end{align*}
Together with $\gamma_n(K)>0$,
\begin{align*}
V_n(K)\geq(\sqrt{2\pi})^n\gamma_n(K)>0.
\end{align*}
Therefore, compact convex set $K$ has nonempty interior in $\mathbb{R}^n$, that is, $K$ is a convex body in $\mathbb{R}^n$.
\end{proof}

\begin{lem}\label{Conti-polar-bound}
Suppose $K_i\in \mathcal{K}^n_o$ with  $\gamma_n(K_i)\geq1/2$ for $i=1,2,\cdots$.
If the sequence $\{K_i\}$ converges to compact convex set $L$ in the Hausdorff metric, then $L\in \mathcal{K}^n_o$.
\end{lem}

\begin{proof}
By the continuity of Gaussian volume,
\begin{align*}
 \gamma_n(L)=\lim_{i\rightarrow+\infty}\gamma_n(K_i)\geq1/2.
\end{align*}
Together with Lemma \ref{interior}, we have $L$ is a convex body in $\mathbb{R}^n$.

Assume that $o\in \partial L$. Then, there exists a $u_0\in S^{n-1}$ such that $h_L(u_0)=0$.
Since $\{K_i\}$ converges to $L$, we have
\begin{align*}
\lim_{i\rightarrow+\infty}h_{K_i}(u_0)=h_L(u_0)=0.
\end{align*}
For arbitrary $\varepsilon>0$, there exists a big enough integer $N_\varepsilon>1$ so that
$h_{K_i}(u_0)<\varepsilon$ for all $i>N_\varepsilon$.
Thus,
$$K_i\subseteq \{x\in\mathbb{R}^n: x\cdot u_0\leq \varepsilon \}$$
for all $i>N_\varepsilon$.
By the fact that $\{K_i\}$ converges to $L$ again, there exists a constant $R>0$ such that $K_i\subseteq B_n(R)$ where $B_n(R)$ is a ball with radius $R$ in $\mathbb{R}^n$.
Hence, for all $i>N_\varepsilon$,
\begin{align*}
K_i\subseteq B_n(R)\cap\{x\in\mathbb{R}^n: x\cdot u_0\leq \varepsilon \}.
\end{align*}

By the following result:
\begin{align*}
\gamma_n(\mathbb{R}^n)=\frac{1}{(\sqrt{2\pi})^n}\int_{\mathbb{R}^n}e^{-\frac{|x|^2}{2}}dx=1,
\end{align*}
we obtain, for halfspace $H^-(u_0)=\{x\in\mathbb{R}^n: x\cdot u_0\leq 0 \}$,
\begin{align*}
\gamma_n(H^-(u_0))=\frac{1}{(\sqrt{2\pi})^n}\int_{H^-(u_0)}e^{-\frac{|x|^2}{2}}dx
=\frac{1}{2}\gamma_n(\mathbb{R}^n)=\frac{1}{2}.
\end{align*}
Together with $\gamma_n(H^-(u_0))=\gamma_n(H^-(u_0)\cap B_n(R))+\gamma_n(H^-(u_0)\backslash B_n(R))$, we have
\begin{align*}
\gamma_n(H^-(u_0)\cap B_n(R))<\frac{1}{2}.
\end{align*}
Hence, for $\varepsilon$ small enough,
\begin{align*}
\gamma_n(K_i)\leq \gamma_n\big( B_n(R)\cap\{x\in\mathbb{R}^n: x\cdot u_0\leq \varepsilon \}\big)<\frac{1}{2},
\end{align*}
for all $i>N_\varepsilon$.
This is a contradiction to the condition $\gamma_n(K_i)\geq1/2$ for $i=1,2,\cdots$.
Therefore, $o$ is an interior point of $L$, that is, $L\in \mathcal{K}^n_o$.
\end{proof}

\begin{remark}\label{rem-1}
From the proof of Lemma \ref{Conti-polar-bound}, we obtain that  $\gamma_n(K)\geq \frac{1}{2}$ for $K\in \mathcal{K}^n_o$ means that the origin $o$ can not be close sufficiently to the boundary $\partial K$.

If $L\in \mathcal{K}^n$ with $o\in \partial L$, then we have $\gamma_n(L)<\frac{1}{2}$.
Hence, the condition ``$\gamma_n(K_i)\geq \frac{1}{2}$ for $K_i\in \mathcal{K}^n_o$" is very necessary for Lemma \ref{Conti-polar-bound}.

\end{remark}

The weak convergence of Gaussian surface area measure is obtained in \cite{HuangXZ}:

\begin{lem}[\cite{HuangXZ}]\label{weak-conve}
Let $K_i\in \mathcal{K}^n_o$ for $i=1,2,\cdots$ such that $\{K_i\}$ converges to $K_0\in \mathcal{K}^n_o$ in the Hausdorff metric, then
$\{S_{\gamma_n,K_i}\}$ converges to $S_{\gamma_n,K_0}$ weakly.
\end{lem}

By the variational formula \eqref{Lp-Gaussian-Measure} of the Gaussian volume $\gamma_n$  for $L_p$ Minkowski combination,
the integral expression of $L_p$ Gaussian surface area measure was obtained in \cite{HuangXZ,Liu,LvWang}.

\begin{lem}[\cite{HuangXZ,Liu,LvWang}]\label{definition-Lp-Gaussian}
Suppose $p\in \mathbb{R}$ and $K\in \mathcal{K}^n_o$. For each Borel set $\eta\subseteq S^{n-1}$, $L_p$ Gaussian surface area $S_{p,\gamma_n}(K,\cdot)$ of $K$ is defined by
\begin{align*}
S_{p,\gamma_n}(K,\eta)
&=\frac{1}{(\sqrt{2\pi})^n}
\int_{\nu^{-1}_K(\eta)}(x\cdot\nu_K(x))^{1-p}e^{-\frac{|x|^2}{2}}d\mathcal{H}^{n-1}(x)\\
&=\int_{\eta}h_K^{1-p}(u)dS_{\gamma_n,K}(u).
\end{align*}
Here, $\nu_K: \partial' K \rightarrow S^{n-1}$ is the Gauss map of $K$
and $\mathcal{H}^{n-1}$ is an ($n-1$)-dimensional Hausdorff measure.
\end{lem}

Since $\{K_i\}$ converges to $K_0$ in the Hausdorff metric for $K_i\in \mathcal{K}^n_o$$(i=0,1,\cdots)$, then $\{h_{K_i}\}$ converges to $h_{K_0}$
uniformly on $S^{n-1}$. Together with the Lemma \ref{weak-conve} and Lemma \ref{definition-Lp-Gaussian}, we obtain the weak convergence of $L_p$ Gaussian surface area measures as follows:

\begin{pro}\label{weak-conve-p}
Suppose $p\in \mathbb{R}$ and $K_i\in \mathcal{K}^n_o$ for $i=0,1,\cdots$.  If $\{K_i\}$ converges to $K_0$ in the Hausdorff metric, then
$\{S_{p,\gamma_n}(K_i,\cdot)\}$ converges to $S_{p,\gamma_n}(K_0,\cdot)$ weakly.
\end{pro}

The uniqueness of the solution to the $L_p$ Gaussian Minkowski problem is obtained in \cite{HuangXZ,Liu}:

\begin{lem}[\cite{HuangXZ,Liu}]\label{uniqueness}
Let $p\geq1$ and $K,L\in\mathcal{K}^n_o$ with $\gamma_n(K),\gamma_n(L)\geq1/2$. If
\begin{align*}
S_{p,\gamma_n}(K,\cdot)=S_{p,\gamma_n}(L,\cdot),
\end{align*}
then, $K=L$.
\end{lem}

The continuity of the solution to the $L_p$ Gaussian Minkowski problem is obtained for $\gamma_n(\cdot)\geq1/2$.
Theorem \ref{thm0-1} is rewritten as Theorem \ref{continuity-1}:

\begin{thm}\label{continuity-1}
Suppose $p\geq1$ and $K_i\in\mathcal{K}^n_o$ with $\gamma_n(K_i)\geq1/2$ for $i=0,1,2,\cdots$.
If the sequence $\{S_{p,\gamma_n}(K_i,\cdot)\}$ converges to $S_{p,\gamma_n}(K_0,\cdot)$ weakly, then the sequence $\{K_i\}$ converges to $K_0$ in the Hausdorff metric.
\end{thm}

\begin{proof}
Assume that the sequence $\{K_i\}$ does not converge to $K_0$.
Then, without loss of generality, we may assume that there exists a constant $\varepsilon_0>0$ such that
\begin{align*}
\|h_{K_i}-h_{K_0}\|\geq \varepsilon_0,
\end{align*}
for all $i=1,2,\cdots$.

By Lemma \ref{Conti-bound}, $\{K_i\}$ is bounded. Thus, from Blaschke selection theorem,
the sequence $\{K_i\}$ has a convergent subsequence $\{K_{i_j}\}$ which
converges to a compact convex set $L_0$.
Clearly, $L_0\neq K_0$.
Together with the continuity of $\gamma_n$ and $\gamma_n(K_{i_j})\geq1/2$ for $i=1,2,\cdots$,
we have
\begin{align*}
 \gamma_n(L_0)=\lim_{j\rightarrow+\infty}\gamma_n(K_{i_j})\geq1/2.
\end{align*}
By $\lim_{j\rightarrow+\infty}K_{i_j}=L_0$ and $\gamma_n(K_{i_j})\geq1/2$ for $i=1,2,\cdots$ again, $L_0\in\mathcal{K}^n_o$ with $L_0\neq K_0$ from Lemma \ref{Conti-polar-bound}.

Since $\{K_{i_j}\}$ converges to $L_0$ in Hausdorff metric, then  $\{S_{p,\gamma_n}(K_{i_j},\cdot)\}$ converges to $S_{p,\gamma_n}(L_0,\cdot)$ weakly by Lemma \ref{weak-conve-p}.
Together with $\{S_{p,\gamma_n}(K_{i_j},\cdot)\}$ converges to $S_{p,\gamma_n}(K_0,\cdot)$ weakly, we have
\begin{align*}
 S_{p,\gamma_n}(K_0,\cdot)=S_{p,\gamma_n}(L_0,\cdot).
\end{align*}
By $\gamma_n(K_0),\gamma_n(L_0)\geq1/2$ and Lemma \ref{uniqueness}, we obtain $K_0=L_0$.
This is a contradiction to $K_0\neq L_0$. Therefore, the sequence $\{K_i\}$ converges to $K_0$ in the Hausdorff metric.
\end{proof}

\vskip 0.5cm
\section{The proof of Theorem \ref{thm0-2}}\label{}
\vskip 0.3cm

In this section, we mainly prove Theorem \ref{thm0-2}. The following lemma will be needed.

\begin{lem}\label{bound-2}
Suppose $p_i\geq1$ and $K_i\in\mathcal{K}^n_o$ with $\gamma_n(K_i)\geq1/2$ for $i=0,1,2,\cdots$.
If $S_{p_i,\gamma_n}(K_i,\cdot)=S_{p_0,\gamma_n}(K_0,\cdot)$  with $\lim_{i\rightarrow+\infty}p_i=p_0$, then the sequence $\{K_i\}$ is bounded.
\end{lem}

\begin{proof}
By $p_i\geq1$ and $\lim_{i\rightarrow+\infty}p_i=p_0$, without loss of generality, we may assume that
\begin{align}\label{pi}
1\leq p_i<2p_0,
\end{align}
for all $i$.
From \eqref{Phi},
 $S_{p_i,\gamma_n}(K_i,\cdot)=S_{p_0,\gamma_n}(K_0,\cdot)$, Lemma \ref{Conti-Min-ineq} and $\gamma_n(K_i)\geq1/2$, we have
\begin{align}\label{Conti-bound-2}
\Phi_{p_i}(K_i)
&=-\frac{1}{p_i\gamma_n(K_i)}\int_{S^{n-1}}
h^{p_i}_{K_i}(u)dS_{p_i,\gamma_n}(K_i,u)+\log\gamma_n(K_i)\nonumber\\
&\geq-\frac{1}{p_i\gamma_n(K_i)}\int_{S^{n-1}}h^{p_i}_{B_n}(u)dS_{p_i,\gamma_n}(K_i,u)+\log\gamma_n(B_n)\nonumber\\
&=-\frac{|S_{p_i,\gamma_n}(K_i,\cdot)|}{p_i\gamma_n(K_i)}+\log\gamma_n(B_n)\nonumber\\
&\geq-\frac{2}{p_i}|S_{p_0,\gamma_n}(K_0,\cdot)|+\log\gamma_n(B_n)\nonumber\\
&\geq-2|S_{p_0,\gamma_n}(K_0,\cdot)|+\log\gamma_n(B_n)
\end{align}

Let
\begin{align*}
R_i=\rho_{K_i}(u_i)=\max\{\rho_{K_i}(u): u\in S^{n-1}\},
\end{align*}
where $u_i\in S^{n-1}$.
Then, $R_iu_i\in K_i$ and $K_i\subseteq R_iB_n$.
Thus,
\begin{align*}
    h_{K_i}(v)\geq R_i(u_i\cdot v)_+,
\end{align*}
for all $v\in S^{n-1}$ and $\gamma_n(K_i)\leq \gamma_n(R_iB_n)$.

Since $S_{p_0,\gamma_n}(K_0,\cdot)$ is not concentrated in any closed hemisphere of $S^{n-1}$ and  $S^{n-1}$ is a compact set, then there exists a constant $c_3>0$ such that
\begin{align*}
    \int_{S^{n-1}}(u\cdot v)_+^{2p_0}dS_{p_0,\gamma_n}(K_0,v)\geq c_3,
\end{align*}
for $u\in S^{n-1}$.
Thus,
\begin{align*}
\Phi_{p_i}(K_i)
&=-\frac{1}{p_i\gamma_n(K_i)}\int_{S^{n-1}}
h^{p_i}_{K_i}(v)dS_{p_i,\gamma_n}(K_i,v)+\log\gamma_n(K_i)\\
&=-\frac{1}{p_i\gamma_n(K_i)}\int_{S^{n-1}}
h^{p_i}_{K_i}(v)dS_{p_0,\gamma_n}(K_0,v)+\log\gamma_n(K_i)\\
&\leq-\frac{R_i^{p_i}}{p_i\gamma_n(K_i)}\int_{S^{n-1}}(u_i\cdot v)_+^{p_i}dS_{p_0,\gamma_n}(K_0,v)+\log\gamma_n(K_i)\\
&\leq-\frac{R_i^{p_i}}{2p_0\gamma_n(R_iB_n)}\int_{S^{n-1}}(u_i\cdot v)_+^{2p_0}dS_{p_0,\gamma_n}(K_0,v)+\log\gamma_n(R_iB_n)\\
&\leq-\frac{c_3R_i^{p_i}}{2p_0\gamma_n(R_iB_n)}+\log\gamma_n(R_iB_n).
\end{align*}

Assume that $\{R_i\}$ is not bounded. Without loss of generality,
we may assume that $\lim_{i\rightarrow+\infty}R_i=+\infty$.
Together with \eqref{gamma=1} and \eqref{pi},
we have
\begin{align*}
\Phi_{p_i}(K_i)\leq-\frac{c_3R_i^{p_i}}{2p_0\gamma_n(R_iB_n)}+\log\gamma_n(R_iB_n)
\rightarrow -\infty,
\end{align*}
as $i\rightarrow+\infty$.
This is a contradiction to \eqref{Conti-bound-2}.
Hence, $\{R_i\}$ is bounded, that is, the sequence $\{K_i\}$ is bounded.
\end{proof}

Theorem \ref{thm0-2} is rewritten as the following Theorem \ref{continuity-2}.

\begin{thm}\label{continuity-2}
Suppose $p_i\geq1$ and $K_i\in\mathcal{K}^n_o$ with $\gamma_n(K_i)\geq1/2$ for $i=0,1,2,\cdots$.
If $S_{p_i,\gamma_n}(K_i,\cdot)=S_{p_0,\gamma_n}(K_0,\cdot)$, then the sequence $\{K_i\}$ converges to $K_0$ in the Hausdorff metric as $\{p_i\}$ converges to $p_0$.
\end{thm}

\begin{proof}
Suppose that $\{K_i\}$ does not converge to $K_0$ in the Hausdorff metric.
Then, there exist a constant $\varepsilon_1>0$ and a subsequence of $\{K_i\}$, denoted by $\{K_i\}$ again, such that
\begin{align*}
\|h_{K_i}-h_{K_0}\|\geq \varepsilon_1,
\end{align*}
for all $i=1,2,\cdots$.

By Lemma \ref{bound-2}, $\{K_i\}$ is bounded. By the Blaschke selection theorem,
we have $\{K_i\}$ has a convergent subsequence $\{K_{i_j}\}$ which
converges to a compact convex set $L_0$.
Clearly, $L_0\neq K_0$.
Together with the continuity of $\gamma_n$ and $\gamma_n(K_{i_j})\geq1/2$ for $i=1,2,\cdots$,
we have
\begin{align*}
 \gamma_n(L_0)=\lim_{j\rightarrow+\infty}\gamma_n(K_{i_j})\geq1/2.
\end{align*}
By $\lim_{j\rightarrow+\infty}K_{i_j}=L_0$ and $\gamma_n(K_{i_j})\geq1/2$ for $i=1,2,\cdots$ again, $L_0\in\mathcal{K}^n_o$ with $L_0\neq K_0$ from Lemma \ref{Conti-polar-bound}.

Since $\{K_{i_j}\}$ converges to $L_0$ in Hausdorff metric, then
$h_{K_{i_j}}\rightarrow h_{L_0}$ uniformly and $S_{\gamma_n,K_{i_j}}\rightarrow S_{\gamma_n,K_{0}}$ weakly as $j\rightarrow+\infty$.
Together with $\lim_{i\rightarrow+\infty}p_i=p_0$, we have $\{S_{p_{i_j},\gamma_n}(K_{i_j},\cdot)\}$ converges to $S_{p_0,\gamma_n}(L_0,\cdot)$ weakly.
By $S_{p_i,\gamma_n}(K_i,\cdot)=S_{p_0,\gamma_n}(K_0,\cdot)$,
\begin{align*}
 S_{p,\gamma_n}(K_0,\cdot)=S_{p,\gamma_n}(L_0,\cdot).
\end{align*}
By $\gamma_n(K_0),\gamma_n(L_0)\geq1/2$ and Lemma \ref{uniqueness}, we obtain $K_0=L_0$.
This is a contradiction to $K_0\neq L_0$. Therefore, the sequence $\{K_i\}$ converges to $K_0$ in the Hausdorff metric.
\end{proof}

\begin{remark}
Since Lemma \ref{Conti-polar-bound} plays a vital role in the proofs of Theorem \ref{thm0-1} and Theorem \ref{thm0-2}, then the condition ``$\gamma_n(K_i)\geq \frac{1}{2}$ for $K_i\in \mathcal{K}^n_o$" is necessary for Theorem \ref{thm0-1} and Theorem \ref{thm0-2} by Remark \ref{rem-1}.
\end{remark}

\vskip 0.3 cm


\end{document}